\documentclass[12pt]{amsart}
\usepackage{amssymb,latexsym, amsmath, amscd, array, graphicx}

\swapnumbers
\numberwithin{equation}{section}

\theoremstyle{plain}

\newtheorem{thm}{Theorem}[section]
\newtheorem{theorem}[thm]{Theorem}
\newtheorem{lem}[thm]{Lemma}

\newtheorem{prop}[thm]{Proposition}
\newtheorem{cor}[thm]{Corollary}

\newcommand\theoref{Theorem~\ref}
\newcommand\lemref{Lemma~\ref}
\newcommand\propref{Proposition~\ref}
\newcommand\corref{Corollary~\ref}
\newcommand\defref{Definition~\ref}

\def\secref{Section~\ref}

\theoremstyle{definition}
\newtheorem{defin}[thm]{Definition}

\newtheorem{rem}[thm]{Remark}

\DeclareMathOperator{\cat}{{\mbox{\rm cat$_{\rm LS}$}}}
\DeclareMathOperator{\syscat}{{\rm cat_{sys}}}
\DeclareMathOperator{\stsyscat}{{\rm cat_{stsys}}}

\DeclareMathOperator{\cl}{{\rm cl}}

\DeclareMathOperator{\stsys}{{\rm stsys}}

\DeclareMathOperator{\sysh}{{\rm sysh}}

\DeclareMathOperator{\vol}{{\rm vol}}

\def\Im{\protect\operatorname{Im}}

\def\Int{\protect\operatorname{Int}}

\def\gf{\varphi}
\def\ga{\alpha}
\def\gb{\beta}
\def\gl{\lambda}

\def\Z{{\mathbb Z}}
\def\Q{{\mathbb Q}}
\def\R{{\mathbb R}}
\def\N{{\mathbb N}}

\def\ds{\displaystyle}

\def\1{\hbox{\rm\rlap {1}\hskip.03in{\rom I}}}
\def\Bbbone{{\rm1\mathchoice{\kern-0.25em}{\kern-0.25em}
{\kern-0.2em}{\kern-0.2em}I}}

\def\ds{\displaystyle}
\def\m{\medskip}
\def\ov{\overline}

\def\la{\langle}
\def\ra{\rangle}

\long\def\forget#1\forgotten{} %

\newcommand{\gmetric}{{\mathcal G}}
\newcommand\gm{\gmetric}

\newcommand\ver[1]{\marginpar{\tiny Changed in Ver \VER}}

\date{\today}
\begin{document}

\centerline{Dedicated to Stephen Smale in occasion of his 80th birthday}

\vskip 2cm

\title[Stable Systolic Category]{Stable Systolic Category of Manifolds and the
Cup-length}

\author[A.~Dranishnikov]{Alexander N. Dranishnikov$^{1}$} %
\thanks{$^{1}$Supported by NSF, grant DMS-0604494}

\author[Yu.~Rudyak]{Yuli B. Rudyak}

\address{Alexander N. Dranishnikov, Department of Mathematics, University
of Florida, 358 Little Hall, Gainesville, FL 32611-8105, USA}
\email{dranish@math.ufl.edu}

\address{Yuli B. Rudyak, Department of Mathematics, University
of Florida, 358 Little Hall, Gainesville, FL 32611-8105, USA}
\email{rudyak@math.ufl.edu}



\begin{abstract}
It follows from a theorem of Gromov that the stable systolic
category $\stsyscat M$ of a closed manifold $M$ is bounded from
below by $\cl_{\Q}(M)$, the rational cup-length of $M$ \cite{SGT}.
In the paper we study the inequality in the opposite direction. In
particular, combining our results with Gromov's theorem, we prove
the equality $\stsyscat (M)=\cl_{\Q} M$ for simply connected
manifolds of dimension $\le 7$.
\end{abstract}

\maketitle

\tableofcontents

\section{Introduction}
All manifolds are assumed to be smooth. The stable systolic category $\stsyscat M$ of a closed manifold $M$ is a natural number that measures a complexity of $M$ in the
language of differential geometry though it does not depend on a
Riemannian metric on $M$. This invariant (as well as some
generalizations) was introduced by Katz and Rudyak
\cite{KR1},\cite{SGT} on the basis of the following theorem of
Gromov~\cite{Gr1,SGT}:

\begin{theorem}[Gromov]\label{t:gromov}
Suppose that for a closed $n$-dimensional manifold $M^n$ the cup
product $a_1\cup\dots\cup a_d$ is non-trivial for some $a_i\in
H^{k_i}(M^n;\Q)$ with $k_1+\dots+k_d=n$. Then there is a constant
$C>0$ such that
\begin{equation}\label{eq:main}
\stsys_{k_1}M^n\dots\stsys_{k_d}M^n\le C\vol(M^n)
\end{equation}
for every Reimannian metric on $M^n$.
\end{theorem}

The invariant $\stsyscat M$ is defined as the maximal $d$ such that
equality \eqref{eq:main} holds. Thus, one can say that $\stsyscat
M\ge\cl_{\Q}M$ where $\cl_{\Q}M$ the rational cup-length of $M$.
This paper can be considered as an attempt to check whether
$\stsyscat M$ coincides with the rational cup-length. Though we
obtained some coincidence results in low dimensions, we failed to do
it in the general case. Moreover, we have got convinced that these
two invariants are different. We believe that it would be an
interesting project to construct a manifold distinguishing
$\stsyscat$ and $\cl_{\Q}$. In the cases where the coincidence
$\stsyscat M=\cl_{\Q}M$ is proven (Theorem~\ref{t:reallength} and
Theorem~\ref{t:reallength2}) one can ask about possible inversion of
Gromov's Theorem. It turns out that this is also a challenging
problem which is done only in some partial cases, see
Remark~\ref{inversion}.

\m Here are the formal definitions. Suppose that we are given a
manifold~$M^n$ with a Riemannian metric $\gm$. Let $\Delta^k$ be the
standard $k$-simplex in $\R^{k+1}$. Consider a singular
smooth~$k$-simplex~$\sigma: \Delta^k \to M$,~$1\le k\le n$. Here we
say that $\sigma$ is smooth if it extends to an open neighborhood of
$\Delta^k$ in $\R^{k+1}$. In particular, $\sigma$ is a Lipschitz
map. Note that the pullback~$\sigma^*\gm$ is positive semidefinite.
Thus, we can speak on the volume of~$\Delta^k$ with respect
to~$\sigma^*\gm$, and we set
\begin{equation*}
\vol_k\sigma=\vol(\Delta^k, \sigma^*\gm).
\end{equation*}

\m We define a smooth singular chain to be a formal linear combination of smooth singular simplices. Given a real or integral singular smooth~$k$-cycle~$c=\sum_i
r_i\sigma_i$, we define the volume of~$c$ as
\begin{equation}
\label{eq:vol}
\ds\vol_k(c):= \sum_i|r_i|\vol_k(\sigma_i).
\end{equation}
For an {\em integral\/} homology class~$\ga$, we define the
volume~$\vol(\alpha)$ by setting
\begin{equation}
\label{eq:integral}
\vol(\ga) = \inf \left\{ \vol_k(c) \mid c\in C_k(M), \quad
[c]=\alpha \right\}
\end{equation}
i.e. the infimum of volumes of {\em integral\/} smooth cycles
representing~$\ga$.  Now we can define the homology systole by setting
\begin{equation*}
\sysh_k(M,\gm)=\inf \left\{ \vol(\ga)\bigm|\ga\in H_k(M)\setminus
\{0\} \right\}.
\end{equation*}
It was originally pointed out by Gromov (see \cite{Be}) that this
definition has a shortcoming.  Namely, for~$S^1\times S^3$ one
observes a ``systolic freedom'' phenomenon, in that the inequality
\begin{equation}\label{e:freedom}
\sysh_1\sysh_3\le C\vol(S^1\times S^3)
\end{equation}
is violated for any~$C\in \R$, by a suitable metric~$\gm$
on~$S^1\times S^3$ \cite{Gr2,Gr3}.
This phenomenon can be overcome by a process of stabilization as
follows.  Given a {\it real} homology class~$\gb\in H_k(M;\R)$, define
the {\it stable norm} of~$\gb$ as
\begin{equation}\label{e:norm}
||\gb||=\inf \left\{ \vol_k(c)\bigm|[c]=\gb \right\}
\end{equation}
where~$c=\sum r_i\sigma_i$,~$r_i\in \R$ is a real singular cycle
representing~$\gb$.  It is known that~$||\cdot ||$ is indeed a norm
on~$H_k(M;\R)$, \cite{Fed, Gr3}.  Furthermore, if~$\ga\in H_k(M)$ is
an integral homology class and~$\ga_{\R}\in H_k(M;\R)$ is the image
of~$\ga$ under the coefficient homomorphism defined by the
inclusion~$\Z \to \R$, then
\begin{equation*}
||\ga_{\R}||=\lim_{k\to\infty}\frac{\vol(k\ga)}{k},
\end{equation*}
where ``$\vol$'' in the numerator is as in \eqref{eq:integral}.  Now we define
the {\it stable~$k$-systole\/} by setting
\begin{equation}
\label{eq:stable} \stsys_k(M,\gm)=\inf\left\{ \|\gb \| \bigm|\gb\in
H_k(M;\Z)_{\R} \setminus \{0\}\right\}
\end{equation}
where $H_k(M;\Z)_{\R}$ denote the image of the coefficient
homomorphism $H_k(M;\Z)\to H_k(M;\R)$.

 \m The definition below is a stable version of systolic category defined
in~\cite{KR1}.

\begin{defin}
\label{d:syscat}
Let~$M$ be a closed~$n$-dimensional manifold, and let~$d\geq 1$ be an
integer.  Consider a partition
\begin{equation}
\label{eq:partition}
n= k_1 + \ldots + k_d,\quad k_1\le k_2\le \cdots \le k_d
\end{equation}
where~$k_i\geq 1$ for all~$i=1,\ldots, d$.  We say that the
partition (or the~$d$-tuple~$(k_1, \ldots, k_d))$ is {\em stable
categorical} for~$M$ if the inequality
\begin{equation*}
\stsys_{k_1}(\gmetric) \stsys_{k_2}(\gmetric) \ldots
\stsys_{k_d}(\gmetric) \leq C(M) \vol_n(\gmetric)
\end{equation*}
is satisfied by all metrics~$\gmetric$ on~$M$, where the
constant~$C(M)$ depends only on the topological type of~$M$, but not
on~$\gmetric$.

The {\em size\/} of a partition $n= k_1 + \ldots + k_d$ is defined
to be the integer~$d$.

The {\em stable systolic category} of~$M$, denoted~$\stsyscat(M)$, is the
largest size of a categorical partition for~$M$.
\end{defin}

In particular, we have~$\stsyscat M \le \dim M$.

\m In the paper we use homotopy theory in order to make upper
estimates of  $\stsyscat M$ by means of the cup-length. For
simply-connected manifolds of dimension $\le 7$ we prove the
equality $\stsyscat M=\cl_{\Q} M$. In the non-simply connected case
we show that $\stsyscat M\le \widetilde{\cl} M$ for closed manifolds
of dimension $\le 5$ where $\widetilde{\cl}M$ denote the twisted
cup-length of $M$. The later result gives one more evidence for
M.~Katz' conjecture connecting the systolic category and the
Lusternik-Schnirelmann category of manifolds by the inequality
$\syscat M\le\cat M$ \cite{KR1,SGT,DKR1,DKR2}. We recall that $\cat
M$ is the minimal $n$ such that $M$ admits a cover by $n+1$ open
sets that are contractible in $M$.

The paper is organized as follows. In Section~\ref{s:skeleta} we
prove some auxiliary results from homotopy theory, in
Section~\ref{s:lemmata} we prove some preliminary results on
relations between the systolic category and the cup-length, and in
Section~\ref{s:cuplength} we prove main results.

\section{Deformations of maps of complexes into skeleta}\label{s:skeleta}

\begin{defin}
\label{d:degree}
Let~$f: M^n\to K^n$ be a map of a closed orientable manifold~$M$ to an
$n$-dimensional CW space $K$, and let~$e$ be an open~$n$-cell of~$K$.

Consider the map $h=h(f,e): M \stackrel{f}\longrightarrow K \to
K/(K\setminus e)\cong S^n.$ We assume that~$M$ and~$e$ are oriented.
We put~$\deg_ef=\deg h$ and call~$\deg_ef$  {\em the degree of~$f$
at~$e$}.
\end{defin}

Clearly, if $f, f': M \to K$ are homotopic then $\deg_ef=\deg_ef'$.

\begin{defin} We say that a map~$f: X \to Y$ of path connected spaces is
{\em~$\pi$-surjective} if~$f_*: \pi_1(X) \to \pi_1(Y)$ is an
epimorphism.
\end{defin}

\begin{lem}\label{l:degree}
Let~$f: M^n\to K^n$ be a map of a closed orientable manifold~$M$ to an
$n$-dimensional CW space $K$ with $n>2$. Let $e_1, \ldots, e_q$ be open $n$-cells
in $K$. If $f$ is $\pi$-surjective and $\deg_{e_j} f=0$ for $j=1, \ldots, q$, then~$f$
is homotopic to a map $g: M \to K$ such that $g(M)\cap
e_j=\emptyset$ for $j=1,\ldots, q$.
\end{lem}

\begin{proof} This is an application of the Hopf trick, cf.~\cite[Theorem
4.1]{Eps}. We perform induction on $q$ and construct a map $f_j: M \to K$ such that
$f\cong f_j$ and $f_j(M)\cap e_m=\emptyset$ for $m\le j$.

First, set $j=1$ and denote $e_1$ by $e$. Consider an open subset
$O$ of $e$ such that $\ov O \subset e$. Without loss of generality,
we can assume that $f_{|f^{-1}O} : f^{-1}O \to O$ is smooth. Take a regular point
$y\in O$ with the preimage $f^{-1}(y)=\{x_1,\dots x_{2k}\}$. Let $V$
and $U_1,\dots, U_{2k}$ be neighborhoods of $y$ and $x_1,\dots,
x_{2k}$, respectively, that are homeomorphic to the closed $n$-ball
and such that $f|_{U_i}:U_i\to V$ is a homeomorphism. Without loss
of generality we may assume that for the local degrees
$\deg_{x_i}f=-\deg_{x_{k+i}}f$ for all $i\le k$. The
$\pi$-surjectivity of $f$ ensures that there is a smooth path $\phi$
from $\partial U_1$ to $\partial U_{k+1}$ without self-crossings and
such that $f\phi$ is a null-homotopic loop in $K\setminus \{y\}$.
(Here we use the condition $n>2$.) We may assume that $\phi$ has
nowhere zero derivative and $\Im \phi$ does not intersect $\Int
U_1$, $\Int U_{k+1}$ and $U_m$ for $m\ne 1,k+1$. Then we attach to
the balls $U_1$ and $U_{1+k}$ a tube $T$, a closed regular
neighborhood of $\Im(\phi)$, to obtain a closed $n$-ball $D$ with
$f(\partial D)\subset M \setminus \Int V$ and $D\cap U_m=\emptyset$
for $m\ne 1, k+1$. In view of the degree condition
$\deg_{x_1}f+\deg_{x_{k+1}}f=0$ and the fact that $f\phi$ is
null-homotopic in $K\setminus \{y\}$ we conclude that the map
$f|_{\partial D_i}:\partial D_i\to K \setminus \Int V$ is
null-homotopic. So, $f$ is homotopic to a map $h$ with $h(D)\subset
K\setminus \Int V$ and $h=f$ on $M\setminus D$. Now,
$h^{-1}(y)=\{x_2, \ldots, x_k, x_{k+2},\ldots, x_{2k}\}$, and an
obvious induction completes the proof for $j=1$.

\m Now suppose that we constructed a map $f_{j-1}$ with some $j\le
q$, and construct $f_{j}$. For each $m<j$ choose a point $a_m\in
e_m$. Now for the cell $e_j$ we do the same procedure as above, by
taking regular value $y_j\in e_j$, its inverse images $x_{j,i},
i=1,\ldots, 2k_j$, etc. and, eventually, a disk $D_j$. However,
here we choose path $\phi_{j}$ so that the tubes $f_j(T{j})$ do
not contain $a_m$'s.  Now we argue as in case $j=1$ and get a map
$h: M \to K$ with $h(M)\cap e_j=\emptyset$ and $a_m\notin h(M)$.
Because of the last condition, we can deform $h$ to a map $f_j: M\to
K$ with $f_j(M)\cap e_m=\emptyset$ for $m\le j$. Now put $g=f_q$.

\end{proof}

\begin{lem}\label{l:skeleton} Let~$f : M^n\to K, n>2$ be a~$\pi$-surjective map of
a closed orientable manifold~$M$ to an arbitrary CW space. If
$f_*:H_n(M) \to H_n(K)$ is the zero homomorphism then~$f$ can be
deformed into the $(n-1)$-skeleton of~$K$.
\end{lem}

\begin{proof} The claim follows from \lemref{l:degree} because in this case $\deg_ef=0$ for all $n$-cells $e$.
\end{proof}

\m Given a commutative ring $R$, we recall that the $R$-cup-length
$\cl_R(\alpha)$ of an element $\alpha\in H^*(X;R)$ is the length of
the longest presentation of $\alpha$ as the product of elements of
positive dimensions. In particular, we can speak on the rational
cup-length $\cl_\Q(\ga)$.

\begin{cor}\label{c:cuplength}
Let~$f: M^n\to K, n>2$ be a~$\pi$-surjective map of a closed orientable
manifold~$M$ to a CW space~$K$ with torsion free~$n$-dimensional
homology. Assume that $\cl_{\Q}M<k$, while the $\Q$-vector
space~$H^n(K;\Q)$ is generated by elements of rational
cup-length~$\ge k$. Then~$f$ can be deformed into~$K^{n-1}$.
\end{cor}

\begin{proof} Since~$H^n(M;\Q)$ does not have elements of cup-length~$\ge k$, we
conclude that~$f^*: H^n(K;\Q)\to  H^n(M;\Q)$ is the zero map. Hence, by the
Universal Coefficient Theorem,~$f_*: H_n(M;\Q)\to H_n(K;\Q)$ is the zero map. Thus,~$f_*:
H_n(M)\to H_n(K)$ is the zero map since~$H_n(K)$ is torsion free.
\end{proof}

\m Recall that the {\em twisted cup-length} $\widetilde{\cl}X$ of a
space~$X$ is the maximal~$k$ such there exists a non-trivial product
$$u_1\cup \cdots \cup u_k\in H^*(X; A_1\otimes \cdots \otimes A_k)$$
where each~$A_k$ is a~$\Z\pi_1(X)$-module,~$u_i\in H^*(X;A_i)$
and~$\dim u_i>0$.

\begin{prop}\label{p:torus}
Let~$f:K^r\to T^{(r)}$ be a map of an~$r$-dimensional CW space to
the~$r$-skeleton of the~$q$-torus~$T=T^q$,~$q\ge r>2$. Suppose that
$\widetilde{\cl}K<r$. Then there exists a map $g: K^r \to T^{(r-1)}$ such that $g|_{K^{(r-2)}}=f|_{K^{(r-2)}}$.
\end{prop}

\begin{proof}
Let~$\kappa\in H^r(T;\{\pi_{r-1}(F)\})$ be the first obstruction for
retraction of~$T$ onto~$T^{(r-1)}$, i.e. the first obstruction to a
section of the fibration~$T^{(r-1)}\to T$ with the homotopy fiber
$F$. It suffices to show that~$f^*(\kappa)=0$. Put $\pi=\Z^q$ and
let $b_{\pi}\in H^1(\pi; I(\pi))$ be the Berstein class of
$\pi$,~\cite{Ber,DR}; here $I(\pi)$ is the augmentation ideal of
$\pi$. In view of universality of the Berstein class~$b_{\pi}$, we
see that $\kappa$ is the image of~$(b_{\pi})^r\in
H^r(\pi;I(\pi)^{\otimes r})$ under the cohomology coefficients
homomorphism defined by some homomorphism $I(\pi)^{\otimes
r}\to\pi_{r-1}(F)$,~\cite[Corollary 3.5]{DR},~\cite[Proposition
34]{Sv}. Then~$f^*(\kappa)$ is the image of~$f^*(b_{\pi})^r$ under
the same coefficients homomorphism. Because of the cup-length
restriction, we conclude~$f^*(b_{\pi})^r=0$.
\end{proof}

\section{Some preliminary results}\label{s:lemmata}

\begin{defin}\label{d:range}
Given an  $m$-tuple~$\ov k=(k_1,\dots,k_m)$, $k_1\le \ldots \le k_m$, consider the map
$\phi:\{1,\dots,m\}\to\N$,~$\phi(i)=k_i$. We define the {\em range}
of the~$m$-tuple~$\ov k$ as the ordered image~$r_1<\dots<r_l$
of~$\phi$.
\end{defin}

This definition allows us to avoid working with repeated~$k_i$.

\m We need the following result on incompressibility, cf.~\cite{Bab1, KR1}

\begin{lem}\label{l:compr}
Let~$M^n$ be a closed smooth manifold, let~$k_1+\cdots +k_m$ be a
stable categorical partition of~$n$, and let~$r_1< \cdots <r_l$ be
its range. Suppose that~$H_{r_i}(M;\Q)\ne 0$ for all~$i$. Let $f: M
\to L$ be a map to a polyhedron~$L$ such that~$f_*:H_{r_i}(M;\Q)\to
H_{r_i}(L;\Q)$ is injective for all~$i$. Then~$f$ cannot be deformed
into $L^{(n-1)}$.
\end{lem}

\begin{proof} Suppose the contrary, i.e. that~$f$ is homotopic to a simplicial
map~$g:M \to L$ with~$g(M)\subset L^{(n-1)}$. We consider simplicial
embeddings $M\to \R^a$ and~$L\to R^b$ and define the embedding
$$
i:M\to M\times L\to \R^a\times \R^b,
$$
where the first map has the form~$m\mapsto (m, g(m))$. The standard
metric on~$\R^{a+b}$ yields a PL metric~$\gm$ on~$M=i(M)$, while the
metric on~$\R^b$ yields a metric~$\gm'$ on~$L$. Note that the
map~$g: (M,\gm) \to (L,\gm')$ does not increase the distance. So,
since~$g_*:H_{r_i}(M;\Q)\to H_{r_i}(L;\Q)$ is injective, we conclude
that~$\stsys_{k_i}(M,\gm)\ge\stsys_{k_i}(L,\gm')>0$. Now we rescale
the metric on $\R^b$ via multiplication by~$t\in \R$. This gives us
a new metric on~$\R^{a+b}$ and also new metrics~$\gm_t,\gm'_t$
on~$M, L$ respectively. Clearly, $\stsys_k(L,
\gm'_t)=t^k\stsys_k(L,\gm')$. Hence,
\[
\stsys_k(M,\gm_t)\ge \stsys_k(L, \gm't)=t^k\stsys_k(L,\gm).
\]
In other words,~$\ds\prod_{i=1}^m\stsys_{k_i}(M,\gm_t)$ grows at
least as fast as $t^n$.

\m On the other hand, if an~$n$-dimensional simplex~$\Delta$ in
$\R^{a+b}$ projects into~$(n-1)$-dimensional simplex in~$\R^b$,
then~$\vol(\Delta, \gm_t)$ grows at most as $t^{n-1}$.
Hence,~$\vol(M,\gm_t)$ grows at most as~$t^{n-1}$. Thus, for
any~$C$, the inequality
\[
\prod_{i=1}^m\stsys_{k_1}(M,\gm_t)\le C\vol(M,\gm_t)
\]
is violated for $t$ large enough. Thus,~$k_1+\cdots +k_m$ is not a stable
categorical partition.
\end{proof}

Let $r_1<\ldots <r_l$ be the range of an $m$-tuple $(k_1, \cdots, k_m)$.
Put~$L_i=S^{r_i}$ for~$r_i$ odd. For~$r_i$ even let~$L_i$ be a CW
complex that is homotopy equivalent to~$\Omega S^{r_i+1}$ and has exactly
one cell in each dimension~$sr_i, s=0, 1, \ldots$ and no other
cells. The existence of such cellular structure on the homotopy type~$\Omega S^{r_i+1}$ follows from the Morse theory~\cite{Mi}, or one can use the James construction~$JS^{r_i}$ \cite{Ha}. Note also that each $r$-cell of~$L_i$ yields a cycle in~$C_r(L_i)$ since its
boundary belongs to the skeleton of dimension $<r-2$. Thus, the
following fact holds true.

\begin{prop}\label{p:torsion}
For each~$k$ we have~$H_k(L_i)=C_k(L_i)$. Furthermore, every subcomplex of~$L_i$
has torsion free homology.
\end{prop}

\begin{lem}\label{l:mono}
Let~$r_1<\cdots <r_l$ be the range of an $m$-tuple~$(k_1, \ldots,
k_m)$. Given a finite CW space~$X$, let~$b_i$ denote the~$r_i$-th
Betti number~$b_{r_i}(X)$. Then for each $i$ there exist
a~$\pi$-surjective map~$f_i:X\to (L_i)^{b_i}$ that induces an
isomorphisms $f_*:H_{r_i}(X;\Q)\to H_{r_i}((L_i)^{b_i};\Q)$. In
particular, the map $f=(\prod f_i)\circ \Delta$ is $\pi$-surjective
and induces a monomorphism
$$
f_*:H_{r_i}(X;\Q)\to
H_{r_i}\left(\prod_{i=1}^{l}(L_i)^{b_i};\Q\right)
$$
for all~$i$ where $\Delta: X \to X^l$ is the diagonal.
\end{lem}

\begin{proof} Note that~$\pi_k(L_i)\otimes \Q=0$ for~$k\ne r_i$~\cite{Serre}. So, the
rationalization of~$L_i$ is the Eilenberg--MacLane space~$K(\Q,r_i)$. Hence, there exists a map~$f_i: X \to (L_i)^{b_i}$ such that~$\gf_*: H_{r_i}(X;\Q) \to H_{r_i}((L_i)^{b_i};\Q)$ is an isomorphism. This map is also $\pi$-surjective, since $L_i$ is simply connected for $r_i>1$ and is a circle otherwise.
\end{proof}

\begin{lem}\label{l:obstruction}
Let~$r_1<\cdots <r_l$ be the range of an $m$-tuple~$(k_1, \ldots,
k_m)$. Let $L=\prod_{i=1}^n(L_i)^{a_i}$, $a_i\in \N$, be equipped with the product
CW complex structure and let~$M^n$ be a closed
orientable~$n$-manifold with $\cl_{\Q}M^n<m$. Suppose that~$\sum
\gl_i\ge m$ for every~$l$-tuple~$(\gl_1,\dots, \gl_l)$ of non-negative
integers with~$\sum \gl_ir_i=n$  and~$\gl_i\le a_i$ for~$r_i$ odd. Then
every~$\pi$-surjective map~$f:M \to L$ is homotopic to a map to
the~$(n-1)$-skeleton of~$L$.
\end{lem}

\begin{proof} Since~$H^*(L;\Q)$ is a free
skew-commutative algebra, we conclude that each element
of~$H^n(L;\Q)$ is a sum of monomials~$x_1\cdots x_l$ where~$x_i\in
H^{\gl_ir_i}((L_i)^{a_i};\Q)$, with~$\sum_{i=1}^l \gl_ir_i=n$, $\gl_i\in \Z_+$. Each
monomial~$x_1\cdots x_l$ can be decomposed as the product of at
least~$\sum \gl_i\ge m$ factors, i.e. it has the cup-length~$\ge m$.
Now the result follows from \corref{c:cuplength} in view of
\propref{p:torsion}.
\end{proof}

\begin{cor}\label{c:s}
Let $M, L$ and $(k_1, \ldots, k_m)$ be as in Lemma $\ref{l:obstruction}$. Suppose that~$\sum
s_j\ge m$ for every~$m$-tuple~$(s_1,\dots, s_m)$ of non-negative
integers with~$\sum s_jk_j=n$  and~$s_j\le a_j$ for~$r_j$ odd. Then
every~$\pi$-surjective map~$f:M \to L$ is homotopic to a map to
the~$(n-1)$-skeleton of~$L$.
\end{cor}

\begin{proof} Put $\gl_i=\sum_{k_j=r_i}s_j$, and the result follows
from \lemref{l:obstruction}.
\end{proof}

\begin{defin}\label{d:sub}
Let~$r_1<\cdots <r_l$ be the range of an $m$-tuple~$(k_1, \ldots,
k_m)$. Let~$\{L_{\ga}, \ga\in A\}$ be a finite family of complexes each of which is
homeomorphic to some~$L_i$ and equip $L=\prod L_{\ga}$ with the product
CW complex structure. Put~$r_{\ga}=r_i$ if~$L_{\ga}$ is
homeomorphic to $L_i$. Let $L=\prod L_{\ga}$. Given~$k\in \N$, we
define a subspace~$L_{[k]}$ of~$L$ by setting
\[
L_{[k]}=\bigcup_{\{\gl_{\ga}\}}\left(\prod_A L_{\ga}^{(\gl_{\ga}r_{\ga})}\right)
\]
where~$\{\gl_{\ga}\}$ runs over the families of non-negative integers $\gl_{\ga}$ with
$\ds\sum_{\ga}\gl_{\ga}\le k$.
\end{defin}

\begin{lem}\label{l:trick2}
Let~$r_1<\cdots <r_l$ be the range of a stable categorical $m$-tuple
$(k_1,\ldots,k_m)$ for a closed $n$-manifold~$M$. Put $b_i=b_{r_i}(M)$.
Let~$\{L_{\ga}\}$ be a finite family of
complexes each of which is homeomorphic to some~$L_i, i=1, \ldots,
l$.
Consider a map~$f:M^n\to L_{[m]}$ and compositions
$$
\CD f_i: M^n @>f>> L_{[m]} @>\subset >> L @>\text{\rm proj}>>
(L_i)^{b_i}
 \endCD
 $$
 and assume that ~$(f_i)_*:H_{r_i}(M^n;\Q)\to
H_{r_i}((L_i)^{b_i};\Q)$ are monomorphisms for all~$i$. Then~$f$
cannot be deformed to~$L_{[m-1]}$.
\end{lem}

\begin{proof}  Let $K_i$ be a simplicial complex that is homotopy equivalent to $L_i$
and define $K_{[k]}$ similarly to $L_{[k]}$.  We equip each $K_{i}$
with a PL Riemannian metric~$\gm^i$. Together these metrics yield
the product metric~$\ov\gm$ on~$\prod K_i$ and hence a metric on
$K_{[m]}$. Choose cellular homotopy equivalences $L_i \to K_i$. They
yield a map $\gf: L_{[m]}\to K_{[m]}$, and it is clear that if $f$
can be deformed to $L_{[m-1]}$ then $\gf f$ can be deformed
$K_{[m-1]}$. So, by way of contradiction, suppose that there exists
a PL map~$g: M \to K_{[m-1]}$ that is homotopic to~$\gf f$.
Furthermore, construct $g_i: M \to (K_i)^{b_i}$ in the same way as
we constructed $f_i$ from $f$.

Consider the embedding~$j:M\to M\times K, j(m)=(m, g(m))$, the graph of~$M$. We
take any metric~$\gm^0$ on~$M$; this gives us the product metric on~$M\times K_{[m]}$,
and this product metric induces a metric~$\gm$ on~$M=j(M)$.

Now we rescale the metric~$\gm^i$ on~$(K_i)^{b_i}$ by multiplication
by $t^{1/r_i}, t\in \R_+$ and get the metric~$\gm^i_t$.
 The metrics~$\gm^i_t$ together with the
metric~$\gm^0$ on~$M$ yield a metric~$\gm_t$ on $M=j(M)$.

Note that $\stsys_{r_i}(K_i,\gm^i_t)$ as a function of~$t$ grows
as~$t$.

Since the map~$f_i:(M,\gm_t)\to ((K_i)^{b_i}, \gm^i_t)$ does not
increase the distance, and since~$(f_i)_*$ is a monomorphism in
dimension~$r_i$, we conclude that $\stsys_{r_i}(M,\gm_t)\ge
\stsys_{r_i}(K_i, \gm^i_t)$. Hence, if $k_i=r_j$ then
$\stsys_{k_i}(M,\gm_t)\ge \stsys_{k_i}(K_j, \gm^j_t)$. Thus,
$\ds\prod_{i=1}^m\stsys_{k_i}(M, \gm^t)$ grows at least as $t^m$.

On the other hand, the~$q$-volume~$\vol(\Delta^q,\gm^i_t)$ of
a~$q$-simplex $\Delta$ in~$K_i$ grows as~$t^{q/r_i}$. Hence, for the
product, say, $\Delta^{q_1}\times \cdots \times \Delta^{q_{m-1}}$ in
$K_{[m-1]}$,~$\Delta^{q_i}\subset K_i$, its volume grows
as~$t^{\sum(q_i/k_i)}$. But~$\sum(q_i/k_i)\le m-1$ by
\defref{d:sub}.

\m
Thus, the inequality
\[
\ds\prod_{i=1}^m\stsys_{k_i}(M,\gm_t)\le C\vol(M,\gm^t)
\]
cannot be true for all~$t$, since its left-hand part grows as at least~$t^m$
while its
right-hand part grows as at most~$t^{m-1}$.
\end{proof}

Every cell~$e$ of some~$L_i$ gives a chain in~$C_*(L_i)$, and this chain is a cycle,
cf. \propref{p:torsion}. We denote by~$[e]$ the corresponding homology
class.

\begin{lem}\label{l:kunneth}
Put~$X=L_i, Y=L_j$. Given a commutative ring $R$, let~$c\in H^k(X;R)$ and $c'\in
H^l(Y;R)$. Let~$e$ and~$e'$ be oriented cells of
dimensions~$k$ and~$l$ of~$X$ and~$Y$, respectively. Then
in~$H^*(X\times Y;R)$ we have:
\begin{equation}\label{eq:tensor}
\la( p_1^*(c)\cup p_2^*(c'),[e\times e']\ra =\la c,[ e]\ra \la c',[e']\ra.
\end{equation}
where~$p_1: X \times Y \to X$ and~$p_2: X \times Y \to Y$ are the projections.
\end{lem}

\begin{proof} See \cite[p. 278]{Ha}, where the equivalence of two definitions of
tensor
 products~$c\otimes c'$ (left-hand and right-hand part of \eqref{eq:tensor},
respectively) is proved.
\end{proof}

\begin{lem}\label{l:cocycle}
Let $R$ be a commutative ring. Let~$(k_1,\ldots ,k_m)$ be
an~$m$-tuple with the range~$r_1<\dots<r_l$. Let~$\{L_{\ga}\}$ be a
finite family of complexes each of which is homeomorphic to
some~$L_i, i=1, \ldots, l$. Consider a $d$-cell

\begin{equation}\label{eq:cell}
e=\prod e^{\gl_{\ga}r_{\ga}}\subset \prod L_{\ga}, \quad \gl_{\ga}\ge 0,\quad
\dim
e=\sum_{\ga}\gl_{\ga}r_{\ga}=d.
\end{equation}

Let~$c\in H^d(\prod L_{\ga};R)$ be such that~$\la c,[e]\ra=1$
and~$\la c,[e']\ra=0$ for all~$d$-cells~$e'$ with~$e\ne e'$.
Then~$\cl_R(c)\ge \sum \gl_{\ga}$.
\end{lem}

\begin{proof}
For~$\gl_{\ga}\ne 0$, let~$c_{\ga}\in H^*(L_{\ga})$ be the
cohomology  class $c_{\ga}:H_*(L_{\ga})\to\Q$ such
that~$c_{\ga}[e^{\gl_{\ga}r_{\ga}}]=1$ and~$c[e']=0$ for all other
cells~$e'$. Since~$H^*(L;R)$ is a free skew-commutative algebra, we
conclude that~$\cl(c_{\ga})=\gl_{\ga}$.

We order all~$\ga$'s with~$\gl_{\ga}\ne 0$,~$\ga_1< \cdots <\ga_s$.
We assume that~$e^{\gl_{\ga}r_{\ga}}$ is oriented and equip~$e$ with
the product orientation. Let~$p_{\ga}: \prod L_{\ga}\to L_{\ga}$ be
the projection and put~$z_{\ga}=p_{\ga}^*c_{\ga}$.
Put~$z=z_{\ga_1}\cup \cdots \cup z_{\ga_s}\in H^*(L;\Q)$. Then, by
\lemref{l:kunneth},~$\la z,[e]\ra=1$ and~$\la z,[e']\ra=0$ for~$e\ne e'$.
Hence,~$z=c$. But, clearly,~$\cl_R(z)\ge\sum\gl_{\ga}$.
Thus,~$\cl_R(c)\ge\sum\gl_{\ga}$.
\end{proof}

\begin{lem}\label{l:obstruction2}
Let~$M^n$ be a closed orientable~$n$-manifold with~$\cl_{\Q}M< m$,
let~$\ov k=(k_1,\ldots ,k_m)$ be an~$m$-partition of~$n$ with the
range~$1=r_1<\dots<r_l$, and let~$\{L_{\ga}\}$ be a finite family of
complexes each of which is homeomorphic to some~$L_i, i=1, \ldots,
l$. Let $K$ be a subcomplex of $L=\prod L_{\ga}$. Suppose
that~$K^{(n)}\subset L_{[m]}$. Then every $\pi$-surjective map~$f: M^n \to K$ is homotopic to a map~$g:M^n\to L_{[m-1]}$.
\end{lem}

\m Note that the condition ~$K^{(n)}\subset L_{[m]}$ holds if and
only if for every~$n$-cell $e$ of $K$ that is defined by
an~$l$-tuple~$(\gl_1, \ldots, \gl_l)$ of non-negative integers, we
have~$\sum \gl_i\le m$.

\begin{proof} For $n=2$ the Lemma is vacuous, so we consider $n>2$. We can assume that~$f: M^n \to K$ is cellular, and so $f(M^n)\subset L_{[m]}$. First we show that $K^{(n-1)}\subset L_{[m-1]}$. Indeed, if $\sum_{i=1}^l \gl_ir_i\le n-1$, then
$(\gl_1+1)r_1+\sum_{i=2}^l\gl_ir_i\le n$, since $r_1=1$. The
condition $K^{(n)}\subset L_{[m]}$ implies that
$\gl_1+1+\sum_{i=2}^m \gl_i\le m$, i.e., $\sum s\gl_i\le m-1$. Thus,
the $(n-1)$-cell defined by $(\gl_1,\dots,\gl_l)$ lies in
$L_{[m-1]}$.

Consider an~$n$-cell~$e=\prod e^{\gl_{\ga}r_{\ga}}\subset L$. Because of \lemref{l:degree}, it suffices to prove that $\deg_ef=0$ if~$e$ is not
in~$L_{[m-1]}$. Indeed, in this case~$\sum \gl_{\ga}\ge m$. Consider
the map
\[
h: M \stackrel{f}\longrightarrow L^{(n)}\stackrel{q}\longrightarrow
L^{(n)}/(L^{(n)}\setminus e)\cong S^n.
\]
Let~$u$ be the generator of~$H^n(S^n)$. Then~$q^*(u)(e)=1$
and~$q^*(u)(e)=0$ for~$e\ne e'$. Hence,~$\cl_{\Q}(q^*(u))\ge m$ by
\lemref{l:cocycle}, and so~$f^*q^*(u)=0$ since~$\cl_{\Q}(M)<m$.
Thus,~$\deg _e f=0$.
\end{proof}

\section{The cup-length and stable systolic category, and LS
category,}\label{s:cuplength}

Here we keep the notation of \secref{s:lemmata}.

\begin{thm}\label{t:reallength}
Let~$M$ be a closed orientable~$n$-manifold and let $k_1+\cdots
+k_m=n$, $k_1\le k_2\le\dots\le k_m$ be a stable categorical
partition for~$M$ with $(m-1)k_m < n$. Then $\cl_{\Q}M\ge m$.
\end{thm}

\begin{proof} By way of contradiction, suppose
that $\cl_{\Q}(M)<m$. In view of \lemref{l:compr}, \lemref{l:mono},
and \corref{c:s}  it suffices to show that~$\sum s_j\ge m$
for every~$m$-tuple~$(s_1,\dots, s_m)$ with~$\sum s_jk_j=n$, $s_j\ge
0$. Indeed, take a map $f$ as in \lemref{l:mono} and note that, by
\corref{c:s}, $f$ can be deformed into $L^{(n-1)}$. But
this contradicts \lemref{l:compr}.

So, we assume that $\sum s_jk_i=n$, $s_j\ge 0$. Then the
inequalities
$$
\sum s_j\ge\sum s_j\frac{k_j}{k_m}=\frac{n}{k_m}> m-1
$$
imply that~$\sum s_j\ge m$.
\end{proof}

\begin{cor}
[\rm\cite{Br}]\label{c:brunn} If the partition $k_1+\dots+k_m=n$
with~$k_1=\dots=k_m$ is stable categorical then $\cl_{\Q}M\ge m$.
\end{cor}

\begin{proof} Since $mk_m=n$, we conclude that the condition of \theoref{t:reallength}
holds.
\end{proof}

\begin{cor}\label{c:2}
If the stable systolic category of~$M$ is equal to~$2$, then
$\cl_{\Q}M=2$.
\end{cor}

\begin{proof}
We note that the condition  of \theoref{t:reallength} holds true for
$m=2$: $k_2<n=k_1+k_2$.
\end{proof}

\begin{cor}\label{c:3}
If~$(k_1,k_2,k_3$) is a stable categorical triple with~$k_3\ne
k_1+k_2$, then $\cl_{\Q}M\ge 3$.

In particular, if~$\stsyscat M^n=3$ and~$n$ is odd then
$\cl_{\Q}M=3$.
\end{cor}

\begin{proof}
Let~$\ov s=(s_1,s_2,s_3)$ satisfy~$s_1k_1+s_2k_2+s_3k_3=k_1+k_2+k_3$. It suffices to show that~$s_1+s_2+s_3\ge 3$. If~$\ov s\ne(1,1,1)$, then at least one~$s_i=0$. If only one~$s_i=0$, then the other two cannot be just ones. Thus~$s_1+s_2+s_3\ge 3$. If~$3>s_i\ne 0$ and~$s_j=0$ for~$j\ne i$, then~$s_i=2$. This implies that~$k_i=\sum_{j\ne i}k_j$, and thus~$i=3$.
\end{proof}

\begin{rem}\label{r:sysh}
Let~$H$ be the 3-dimensional Heisenberg manifold, see
e.g.~\cite{TO}. Since~$H$ is a~$K(\pi,1)$-space, we conclude
that~$\syscat M=3$ by~\cite{Gr1}. On the other hand,~$\stsyscat M=2$
by \corref{c:3}, since the cup-length of~$H$ is equal to 2.
\end{rem}

\begin{thm}\label{t:8}
For simply connected closed manifolds of dimension $< 8$ the stable
systolic category coincides with the rational cup-length, $\stsyscat
M=\cl_{\Q}M$.
\end{thm}

\begin{proof}
In this case~$\stsyscat M<4$, and the result follows from
Corollaries~\ref{c:brunn}, \ref{c:2}, \ref{c:3}.
\end{proof}

The following is a generalization of Theorem~\ref{t:reallength}.

\begin{thm}\label{t:reallength2}
Let~$M$ be a closed orientable~$n$-manifold. Suppose that the
partition~$k_1+\cdots +k_m$ of~$n$ is stable categorical for~$M$.
Put~$b_i=b_{k_i}(M)$. Suppose that there exists $d\le m$ such
that~$\sum_{i=1}^{d}k_i>(d-1)k_l$. Additionally assume that $b_i=1$
and~$k_i$ is odd for $i> d$.

Then $\cl_{\Q}M\ge m$.
\end{thm}

\begin{proof} Suppose
that $\cl_{\Q}(M)<m$. As in Theorem~\ref{t:reallength} it suffices
to show that~$\sum s_i\ge m$ for every~$m$-tuple~$(s_1,\dots, s_m)$
with~$\sum s_ik_i=n$, $s_i\ge 0$,~$s_i\le b_i$ for~$k_i$ odd.
Because of our assumption,~$s_i\le 1$ for~$i>l$.

Define~$J_s=\{i\mid s_i=s, i>d\}$. Note that~$\sum_{i=1}^m
k_i=\sum_{i=1}^m s_ik_i$. We subtract~$\sum_{i=1}^m k_i$
from~$\sum_{i=1}^m s_ik_i$ and obtain the equality
\begin{equation}\label{eq:j0}
\sum_{i=1}^d(s_i-1)k_i=\sum_{i\in J_0}k_i.
\end{equation}
Now we obtain the chain of inequalities :
\begin{equation*}
\begin{aligned}
\sum_{i=1}^ds_i & \ge\sum_{i=1}^ds_i\frac{k_i}{k_d}=
\sum_{i=1}^d\frac{k_i}{k_d}+\sum_{i\in J_0}\frac{k_i}{k_d}>
d-1+|J_0|.
\end{aligned}
\end{equation*}
Hence $\ds\sum_{i=1}^ds_i\ge d+|J_0|$ and
$
\sum_{i=1}^m s_i=\sum_{i=1}^ds_i+|J_1|\ge d+|J_0|+|J_1|=m.
$
\end{proof}

\begin{thm}\label{t:ones}
Let~$(1,...,1,k)$ be a stable categorical~$m$-tuple for a closed
orientable manifold~$M^n$ with $b_1(M)\ge 3$. Then the twisted cup-length
$\widetilde{\cl}\,M$ of $M$ is more than or equal to $m$.
\end{thm}

\begin{proof}  We have~$n=m-1+k$. Put~$a=b_1(M)$ and~$b=b_k(M)$.  In view of
Corollary~\ref{c:brunn} we may assume that~$k>1$.  Consider a~$\pi$-surjective
map~$h:M\to T^a\times L_m^b$ that induces an isomorphisms in
1-dimensional homology and a monomorphism in~$k$-dimensional
rational cohomology, see \lemref{l:mono}. By way of contradiction,
assume that the twisted cup-length of~$M$ is less than~$m$, and so
less that~$n-1$. Put $K=(T^a)^{(n-1)}\times L_m^b$. Because of \propref{p:torus}, there exists a~$\pi$-surjective map~$f: M \to K$ that induces an isomorphisms in
1-dimensional homology and a monomorphism in~$k$-dimensional
rational cohomology.

\m It suffices to prove that~$K^{(n)}\subset L_{[m]}$, i.e. that
every~$n$-cell~$e$ in~$K$ is contained in~$L_{[m]}$. Indeed, in this case $f$ deforms into $L_{[m-1]}$ by \lemref{l:obstruction2}. But this contradicts \lemref{l:trick2}.

\m
An~$n$-cell~$e$ is defined by an $(a+b)$-tuple~$(\gl_1,\dots,\gl_a,
\gl_{a+1}, \ldots, \gl_{a+b}$) such that~$\gl_i\ge 0$ and
\begin{equation}
\label{eq:sum=n}
\sum_{i=1}^a \gl_i+ \sum_{j=1}^b k\gl_{a+j}=n,
\end{equation}
see \eqref{eq:cell}.
If~$e$ is not in~$L_{\{m\}}$, then~$\ds\sum_{i=1}^{a+b}\gl_i\ge m+1$.
Hence, by subtracting, we get
\begin{equation}\label{eq:sum=0}
(k-1)\sum_{j=1}^{b}\gl_{a+j}\le n-m-1=k-2.
\end{equation}
So, since~$k>1$, we conclude that~$\sum_{j=1}^b \gl_{a+j}=0$, i.e.~$e$
should be a product of 1-cells from~$(T^a)^{(n-1)}$. But this is
impossible.
\end{proof}

\begin{thm}
\label{t:dim=5} For a closed~manifold~$M$ with~$\dim M\le 4$ we have $\stsyscat M=\cl_{\Q}(M)$. Furthermore, for $\dim =5$ there
are the inequalities $$\stsyscat M\le\widetilde{\cl}M\le\cat M.$$
\end{thm}

\begin{proof} The first claim follows from \corref{c:2}, \corref{c:3} and
\corref{c:brunn}

For the second claim, the second inequality is well-known, so we
prove the first one. The case~$m=2$ is covered by \corref{c:2}. The
case~$m=3$ is covered by \corref{c:3}. The case~$m=5$ follows from
\corref{c:brunn}. The case~$m=4$ is covered by\theoref{t:ones} since
the only possibility for the systolic 4-tuple is:~$\bar
k=(1,1,1,2)$.
\end{proof}

\begin{rem}\label{inversion} One can try to invert \theoref{t:gromov}
as follows: Let $k_1\le \ldots\le k_m$ be a stable categorical
partition for $M$ and assume that $\cl_{\Q}(M)=m$. Do there exist
$a_i\in H^{k_i}(M;\Q)$ such that $a_1\cup \cdots \cup a_m\ne 0$? We
do not know if this is true in general, although this holds true in
some special cases ($m=2$, or $k_1=k_m$, etc.)
\end{rem}

\end{document}